\documentclass{amsart}
\usepackage{amsmath}
\usepackage{amssymb}
\usepackage{graphicx}
\usepackage[symbol]{footmisc}
\usepackage{color}
\usepackage{amsthm,enumitem}
\newtheorem{theorem}{Theorem}[section]

\newtheorem{lemma}[theorem]{Lemma}
\newtheorem{corollary}[theorem]{Corollary}
\newtheorem{proposition}[theorem]{Proposition}
\theoremstyle{definition}

\newtheorem{remark}{Remark}
\newenvironment{Proof}{{\textit{Proof}.}\ }{~$\square$\vspace{0.2truecm}}

\newcommand{\Aut}{\mbox{\rm Aut}}

\newcommand{\Z}{\mathbb{Z}}
\newcommand{\C}{\mathbb{C}}

\newcommand{\M}{{\rm M}}
\newcommand{\I}{{\rm I}}
\newcommand{\E}{{\rm E}}
\newcommand{\GL}{{\rm GL}}
\newcommand{\SL}{{\rm SL}}
\newcommand{\Gal}{{\rm Gal}}

\newcommand{\core}{\mbox{\rm Core}}

\begin{document}

\title[Locally solvable and solvable-by-finite maximal subgroups]{Locally solvable and solvable-by-finite\\  maximal subgroups of $\GL_n(D)$}
\author[Huynh Viet Khanh]{Huynh Viet Khanh$^{1, 2}$}
\thanks{This work is funded by  Vietnam National University HoChiMinh City (VNUHCM) under grant number B2020-18-02}

\author[Bui Xuan Hai]{Bui Xuan Hai $^{1,2}$}
\address[1]{Faculty of Mathematics and Computer Science, University of Science, Ho Chi Minh City, Vietnam;}
\address[2]{Vietnam National University, Ho Chi Minh City, Vietnam.}
\email{ huynhvietkhanh@gmail.com; bxhai@hcmus.edu.vn}

	
	\maketitle
	
\begin{abstract} 
	This paper aims at studying solvable-by-finite and locally solvable maximal subgroups of an almost subnormal subgroup of the general skew linear group $\GL_n(D)$ over a division ring $D$. It turns out that in the case where $D$ is non-commutative, if such maximal subgroups exist, then either it is abelian or $[D:F]<\infty$. Also, if $F$ is an infinite field and $n\geq 5$, then every locally solvable maximal subgroup of a normal subgroup of $\GL_n(F)$ is abelian.
\end{abstract}

\section{Introduction and statements of main results}

The question of the existence of maximal subgroups in a division ring is difficult, and it has not been settled completely (see \cite{akbari}, \cite{dorbidi2011}, \cite{Haz-Wad}). Despite these, various aspects of maximal subgroups in division rings have been studied substantially by many authors. Let $D$ be a division ring and $D^*=D\backslash\{0\}$ its multiplicative subgroup. In \cite{ebrahimian_04}, it was showed that every nilpotent maximal subgroup of $D^*$ is abelian. In \cite{nassab14}, this result was extended to nilpotent maximal subgroups of a subnormal subgroup $G$ of $D^*$. However, we do not have an analogous result if the word ``nilpotent'' is substituted by ``solvable". In fact, the set $\C^*\cup \C^* j$ forms a non-abelian metabelian maximal subgroup of the multiplicative group $\mathbb{H}^*$ of the division ring of real quaternions $\mathbb{H}$ (see \cite{akbari}). More generally, it was shown that if  a subnormal subgroup $G$ of $D^*$ contains a non-abelian solvable maximal subgroup, then $D$ must be a cyclic algebra of prime degree over the center of $D$ (see \cite{hai-khanh} or \cite{moghaddam}). Recently, we have obtained parallel results in \cite{khanh-hai_almot} for locally nilpotent and locally solvable maximal subgroups of an almost subnormal subgroup of $D^*$. 

Continuing in this direction, we devote Section 2 of the present paper to  examining solvable-by-finite maximal subgroups of an almost subnormal subgroup of the general skew linear group $\GL_n(D)$, for an integer $n\geq 1$. Solvable-by-finite skew linear groups were also considered by several authors before (see \cite[Section 2]{hazrat}, \cite{nassab14}, \cite{wehrfritz_07}). In Section 2, among other corollaries, we prove the following theorem.

\begin{theorem} \label{theorem_1.1}
	Let $D$ be a non-commutative division ring with center $F$ which contains at least five elements, and $G$ an almost subnormal subgroup of $\GL_n(D)$ with $n\geq 1$. If $M$ is a non-abelian solvable-by-finite maximal subgroup of $G$, then $[D:F]<\infty$. Furthermore, we have $F[M]=\M_n(D)$, and there exists a maximal subfield $K$ of $\M_n(D)$  containing $F$ such that $K^*\cap G$ is the Fitting subgroup of $M$, $K/F$ is a Galois extension, $N_{\GL_n(D)}(K^*)\cap G=M $, $K^*\cap G \unlhd M$, $M/K^*\cap G\cong\Gal(K/F)$ is a finite simple group of order $n\sqrt{[D:F]}$.
\end{theorem}

Section 3 is about studying locally solvable maximal subgroups of an almost subnormal subgroup of $\GL_n(D)$. As we have mentioned above, the case $n=1$ has been investigated subsequently in \cite{khanh-hai_almot}.  Therefore, we shall focus on the remaining case $n\geq2$ only. In \cite{dorbidi2011}, it was  pointed out that if $D$ is a non-commutative division ring and $n \geq 2$, then every solvable maximal subgroup of $\GL_n(D)$ is abelian. In \cite{khanh-hai} and \cite{moghaddam}, this result was generalized to solvable maximal subgroups of a normal subgroup of $\GL_n(D)$. To the best of our knowledge, the latest results regarding locally solvable maximal subgroups of $\GL_n(D)$ were obtained in \cite{rem-kiani}. Indeed, it was shown in Theorem 1.6 of \cite{rem-kiani} that if $D$ is an infinite division ring and $n\geq 2$, then every locally nilpotent maximal subgroup of $\GL_n(D)$ is abelian. Also, a similar result for locally solvable maximal subgroups was presented in \cite[Theorem 1.5]{rem-kiani} under some additional conditions. Here, we generalize these results to the case of locally solvable maximal subgroups of a normal subgroup of $\GL_n(D)$. More precisely, we prove in Section 3 the following result.

\begin{theorem}\label{theorem_1.2}
	Let $D$ be a non-commutative division ring with center $F$ which contains at least five elements, and $G$ a normal subgroups of $\GL_n(D)$ with $n\geq2$. If $M$ is a locally solvable maximal subgroup of $G$, then $M$ is abelian.
\end{theorem}

Let us return to the linear case; that is, to the case when $D=F$ is a field. It is worth to noting that the two conditions ``local solvability'' and ``solvability'' are equivalent in this case. For an infinite field $F$ and $n<5$, the work \cite{dorbidi2011} demonstrates that $\GL_n(F)$ always contains non-abelian solvable maximal subgroups. Furthermore, in \cite{dorbidi2011} all solvable maximal subgroups of $\GL_2(F)$ for an arbitrary field $F$ have been determined completely. At the end of Section 3, we show that if $n\geq 5$ and $F$ is infinite, then every solvable maximal subgroup of a normal subgroup of $\GL_n(F)$ is abelian. 

\begin{theorem}\label{theorem_1.3}
	Let $F$ be an infinite field, and $G$ a normal subgroup of $\GL_n(F)$ with $n\geq5$. If $M$ is a solvable maximal subgroup of $G$, then $M$ is abelian.
\end{theorem}
The motivation of the obtained results is \cite[Conjeture 1]{akbari}  which states that the general skew linear group $\GL_n(D)$ contains no solvable maximal subgroups if $D$ is non-commutative and $n\geq 2$ or if $D$ is commutative and $n\geq 5$. Additionally, we refer the reader to \cite{dorbidi2011} and \cite{moghaddam} for a full discussion of this conjecture. In view of the obtained results, this conjecture is reduced to the case of abelian maximal subgroups, which is exactly the Conjecture 2 of \cite{akbari}.

Let $R$ be a ring, $S$ a subring of $R$, and $G$ a subgroup of $R^*$ normalizing $S$ such that $R=S[G]$. Suppose that $N=G\cap S$ is a normal subgroup of $G$ and $R=\bigoplus_{t\in T} tS$, where $T$  is some (and hence any) transversal $T$ of $N$ to $G$. Then, we say that $R$ is a crossed product of $S$ by $G/N$ (see \cite{wehrfritz_91} or \cite[1.4]{shirvani-wehrfritz}). Let $K/F$ be a cyclic extension of fields with the Galois group generated by an automorphism $\sigma$ of order $s=\dim_FK$. Fixing a nonzero element $ a \in F$ and a symbol $x$, we let 
$$C=K\cdot1\oplus K\cdot x\oplus\cdots\oplus K\cdot x^{s-1},$$
and multiply elements in $C$ by using the distributive law and the two rules 
$$x^s=a,\;\;\;\;\; x\cdot b=\sigma(b)x,$$
for any $b\in K$. Then $C$ is an $F$-algebra and is called the \textit{cyclic algebra associated with ($K/F, \sigma$)  and $a\in F\backslash\{0\}$} (see \cite[p.218]{tylam_FC}). 

Throughout this paper, we denote by $D$ a division ring with center $F$ and by $D^*$ the multiplicative group of $D$. For a positive integer $n$, the symbol $\M_n(D)$ stands for the matrix ring of degree $n$ over $D$. We identify $F$ with $F\I_n$ via the ring isomorphism $a\mapsto a\I_n$, where $\I_n$ is the identity matrix of degree $n$. If $S$ is a subset of $\M_n(D)$, then $F[S]$ denotes the subring of $\M_n(D)$ generated by the set $S\cup F$. Also, if $S$ is a subset of $D$, then $F(S)$ is the division subring of $D$ generated by $S\cup F$. Recall that a division ring $D$ is \textit{locally finite} if for every finite subset $S$ of $D$, the division subring $F(S)$ is a finite dimensional vector space over $F$. If $A$ is a ring or a group, then $Z(A)$ denotes the center of $A$.

Let $V =D^n= \left\{ {\left( {{d_1},{d_2}, \ldots ,{d_n}} \right)\left| {{d_i} \in D} \right.} \right\}$. If $G$ is a subgroup of $\GL_n(D)$, then $V$ may be viewed as $D$-$G$ bimodule. Recall that a subgroup $G$ of $\GL_n(D)$ is  \textit{irreducible} (resp. \textit{reducible, completely reducible}) if $V$ is irreducible (resp. reducible, completely reducible) as $D$-$G$ bimodule. If $F[G]=\M_n(D)$, then   $G$ is  \textit{absolutely irreducible} over $D$. An irreducible subgroup $G$ is \textit{imprimitive} if there exists an integer $ m \geq 2$ such that $V =  \oplus _{i = 1}^m{V_i}$ as left $D$-modules and for any $g \in G$ the mapping $V_i \to V_ig$ is a permutation of the set $\{V_1, \cdots, V_m\}$. If $G$ is irreducible and not imprimitive, then $G$ is \textit{primitive}. 

\section{Solvable-by-finite maximal subgroups}\label{section_subnormal}

Let us recall the notion of an almost subnormal subgroup. Let $G$ be a group and $H$ a subgroup of $G$. Following Hartley \cite{Hartley_89}, we say that $H$ is an \textit{almost subnormal} subgroup of $G$ if there is a finite chain of subgroups
$$H=H_0\leq H_1\leq\cdots\leq H_r=G,$$
such that either $[H_{i+1}:H_i]$ is finite or $H_i$ is normal in $H_{i+1}$ for $1\leq i\leq r-1$. It is clear that if $H$ is a subnormal subgroup of $G$, then it is almost subnormal. Let $D$ be a division ring and $n$ an integer number. It was shown in \cite[Theorem 3.3]{ngoc_bien_hai_17} that if $n\geq 2$ and $D$ is infinite, then every almost subnormal subgroup of $\GL_n(D)$ is normal. On the other hand, this result does not hold in the case $n=1$; that is, the case $\GL_1(D)=D^*$. There are examples of division rings which contain almost subnormal subgroups that are not subnormal (see \cite{deo-bien-hai-19} and \cite{ngoc_bien_hai_17}). In addition, it is surprising that every non-central almost subnormal subgroup of $D^*$ always contains a non-central subnormal subgroup. This fact allows us to extend naturally a lot of results for subnormal subgroups to those of almost subnormal subgroups. This interesting result was obtained previously in \cite{hai-minh-bien-chua} by using somewhat properties of graphs. In this paper, we give an alternative proof for this fact without using graphs (Proposition \ref{proposition_2.2}). For this purpose, we need the following lemma which may be adopted from \cite[Lemma 4]{stuth}.
\begin{lemma}\label{lemma_2.1}
	Let $D$ be a division ring, and $N$ a non-central subnormal subgroup of $D^*$. If $G$ is a non-central subgroup of $D^*$, which is invariant under $N$, that is  $x^{-1}Gx\subseteq G$ for all $x\in N$, then $G\cap N$ is non-central.
\end{lemma}
 
\begin{proposition}\label{proposition_2.2}
	Let $D$ be a division ring with center $F$. Then, every non-central almost subnormal subgroup of $D^*$ contains a non-central subnormal subgroup.
\end{proposition}
\begin{proof}
	If $G$ is a non-central almost subnormal subgroup of $D^*$, then by definition, there exists a finite chain of subgroups 
	$$G=G_0\leq G_1\leq \cdots\leq G_n=D^*,$$
	for which either $[G_{i+1}:G_i]$ is finite or $G_i\unlhd G_{i+1}$. We will prove that $G$ contains a non-central subnormal subgroup  by induction on $n$. 
	
	\bigskip 
	
	\textbf{\textit{Step 1. }} \textit{To prove the proposition when $n=1$}. Assume that $[D^*:G]$ is finite. 
	If we set $N=\core_{D^*}(G)=\bigcap_{x\in D^*}x^{-1}Gx$, then $N$ is a normal subgroup of finite index in $D^*$. It is clear that $D^*$ is radical over $N$. We assert that $N$ is not contained in $F$. Assume by contradiction that $N\subseteq F$. We divide our situation into two cases:
	
	\bigskip 
	
	\textit{Case 1.1: $F$ is finite.}  Since $F$ is finite, so is $N$. Additionally, we have $D^*$ is radical over $N$, from which it follows that every element of $D^*$ is periodic. According to [13, Theorem 8], we conclude that $D$ is commutative, a contradiction.
	
	\bigskip 
	
	\textit{Case 1.2: $F$ is infinite.} If se set $s=[D^*:N]$, then for any $a, b\in D^*$ we have $a^s, b^s \in N\subseteq F$. This implies that $a^{-s} b^{-s}a^sb^s=1$ for any $a,b\in D^*$. In other words, $D^*$ satisfies the non-trivial group identity $x^{-s} y^{-s}x^sy^s=1$. In view of [21, Theorem 2.2], we obtain that $D$ is commutative, a contradiction.
	
	\bigskip
	
	\textbf{\textit{Step 2.}} \textit{To prove the proposition when $n>1$}. It follows by induction hypothesis that $G_1$ possesses a non-central subnormal subgroup $N$ of $D^*$ with a subnormal series $N=N_0 \unlhd N_1\unlhd\cdots\unlhd N_m=D^*$. There are two possible cases:
	
	\bigskip
	
	\textit{Case 2.1:} $G=G_0  \unlhd  G_1$. In this case, we have $N\cap G \unlhd N\cap G_1=N$; recall that $N\subseteq G_1$. Since $G$ is invariant under $G_1$, it is also invariant under $N$. It follows from Lemma \ref{lemma_2.1} that $N\cap G$ is non-central. Consequently, $N\cap G$ is a non-central subnormal subgroup of $D^*$ contained in $G$, with the subnormal series
		$$ N\cap G \unlhd N=N_0 \unlhd N_1\unlhd\cdots\unlhd N_m=D^*.$$
		
	\bigskip
	
	\textit{Case 2.2:} $[G_1:G]<\infty$. If $ C=\core_{G_1}(G)$, then $C$ is a normal subgroup of finite index in $G_1$ contained in $G$. We claim that $C$ is non-central. Indeed, assume on contrary that $C\subseteq F$. If $F$ is finite, then $C$ is a finite group. This yields that $G_1$ is also finite, and so is $N$. It follows from \cite[Theorem 8]{her} that $N$ is central, a contradiction. In the remaining case where $F$ is infinite, if $C\subseteq F$, then $G_1$ is a non-central almost subnormal subgroup of $D^*$  satisfying  the identity $x^{-k}y^{-k}x^{k}y^{k} = 1$, where $k=[G_1:C]$, which contrasts to \cite[Theorem 2.2]{ngoc_bien_hai_17}. Therefore $C$ is non-central, as claimed. By the same arguments used in Case 2.1, we conclude that $C\cap N$ is a non-central subnormal subgroup of $D^*$.
\end{proof}
\begin{corollary}\label{corollary_2.3}
	Let $D$ be a division ring with center $F$, and $G$ is a non-central almost subnormal subgroup of $D^*$. If $K$ is a division subring of $D$ normalized by $G$, then either $K\subseteq F$ or $K=D$.
\end{corollary}
\begin{proof}
	By Proposition \ref{proposition_2.2}, $G$ contains a non-central subnormal subgroup $N$, which also normalizes $K$. The result follows immediately from \cite[Theorem 1]{stuth}.
\end{proof}

\begin{lemma}\label{lemma_2.4}
	Let $D$ be a division ring with center $F$, and $G$ an almost subnormal subgroup of $D^*$. If $G$ is  (locally solvable)-by-finite, then $G\subseteq F$.
\end{lemma}

\begin{Proof} 
	Let $A$ be a locally solvable normal subgroup of finite index in $G$. Since $A$ is an almost subnormal subgroup of $D^*$, it follows from \cite{khanh-hai_almot} that $A\subseteq F$. This implies that $G/Z(G)$ is finite and, by Schur's Theorem (\cite[Lemma 1.4, p.115]{passman_77}), we conclude that $G'$ is finite. This means that $G'$ is a finite almost subnormal subgroup of $D^*$. If $G'$ is non-central, then by Proposition \ref{proposition_2.2}, $G'$ contains a non-central finite subnormal subgroup of $D^*$, this contradicts  \cite[Theorem 8]{her}. Hence, $G'\subseteq F$, which means that $G$ is solvable. Therefore, the result follows from \cite{khanh-hai_almot}.
\end{Proof}

\begin{lemma}\label{lemma_2.5}
	Let $D$ be a division ring with center $F$, and $G$ an almost subnormal subgroup of $D^*$. Assume that $M$ is a non-abelian solvable-by-finite maximal subgroup of $G$. If $A$ is a normal subgroup of $M$, then either $A$ is abelian or $F(A)=D$.
\end{lemma}
\begin{proof}
	The condition $A \unlhd M$ implies immediately that $F(A)$ is normalized by $M$ and so $M\subseteq N_{D^*}(F(A)^*)\cap G\subseteq G$. By virtue of maximality of $M$, we have either $N_{D^*}(F(A)^*)\cap G=M$ or $N_{D^*}(F(A)^*)\cap G=G$. If the former case occurs, then $F(A)^*\cap G$ is contained in $M$, which shows that $A $ is normal in $F(A)^*\cap G$. Consequently, $A$ is an almost subnormal subgroup of $F(A)^*$ contained in $M$. Since $M$ is solvable-by-finite, so is $A$. It follows from Lemma \ref{lemma_2.4} that $A$ is contained in the center of $F(A)$ and so $A$ is abelian. In the latter case, the division subring $F(A)$ is  normalized by $G$. It follows from Corollary \ref{corollary_2.3} that either $A\subseteq F$ or $F(A)=D$. In either case, we always have $A$ is abelian or $F(A)=D$. Our proof is now finished.	
\end{proof}

\begin{lemma}[{\cite[Proposition 4.1]{wehrfritz_07}}]\label{lemma_2.6}
	Let $D=E(A)$ be a division ring generated by its metabelian subgroup $A$ and its division subring $E$ such that $E\leq C_D(A)$. Set $H=N_{D^*}(A)$, $B=C_A(A')$, $K=E(Z(B))$, $H_1=N_{K^*}(A)=H\cap K^*$, and let $\tau(B)$ be the maximal periodic normal subgroup of $B$.
	\begin{enumerate}[font=\normalfont]
		\item If $\tau(B)$ has a quaternion subgroup $Q=\left\langle i,j\right\rangle $ of order $8$ with $A=QC_A(Q)$, then $H=Q^+AH_1$, where $Q^+=\left\langle Q,1+j,-(1+i+j+ij)/2\right\rangle$. Also, $Q$ is normal in $Q^+$ and $Q^+/{\left\langle -1,2\right\rangle}\cong\Aut Q\cong Sym(4)$.	
		\item If $\tau(B)$ is abelian and contains an element $x$ of order $4$ not in the center of $B$, then $H=\left\langle x+1\right\rangle AH_1$.
		\item In all other cases, $H=AH_1$.
	\end{enumerate}
\end{lemma}

\begin{theorem}\label{theorem_2.7}
	Let $D$ be a division ring with center $F$, and $G$ an almost subnormal subgroup of $D^*$. If $M$ is a non-abelian solvable-by-finite maximal subgroup of $G$, then $M$ is abelian-by-finite and $[D:F]<\infty$.
\end{theorem}

\begin{Proof} 
	It follows from Lemma \ref{lemma_2.5} that $F(M)=D$. Let $N$ be a solvable normal subgroup of finite index in $M$. In the case where $N$ is abelian, then, of course, $M$ is abelian-by-finite and $L=F(N)$ is a subfield of $D$ normalized by $M$. Take a transversal $\{x_1,x_2,\ldots,x_k\}$ of $N$ and set
	$$\Delta=Lx_1+Lx_2+\cdots+Lx_k.$$
	The displayed relation provides that $\Delta$ is a domain with $\dim_L\Delta\leq k$. The implication of this fact is that $\Delta$ is a centrally finite division ring. Since $\Delta$ contains both $F$ and $M$, it must be coincided with $D$. 
	
	To settle the whole theorem, there remains to examine the case where $N$ is a non-abelian. Therefore, we may suppose that $N$ is solvable with derived series of length $s\geq2$. In other words, we have the following series.
	$$1=N^{(s)}\unlhd N^{(s-1)}\unlhd N^{(s-2)}\unlhd\cdots\unlhd N'\unlhd  N \unlhd M.$$
	If we set $A=N^{(s-2)}$, then $A$ is a non-abelian metabelian normal subgroup of $M$. As $A$ is non-abelian, Lemma \ref{lemma_2.5} again says that $F(A)=D$, from which it follows that that $Z(A)=F^*\cap A$ and $F=C_D(A)$. Set $H=N_{D^*}(A)$,  $B=C_A(A')$, $K=F(Z(B))$, $H_1=H\cap K^*$, and $\tau(B)$ to be the maximal periodic normal subgroup of $B$. Then, it is a simple matter to check that $H_1$ is an abelian group and $\tau(B)$ is characteristic in $B$. It follows from \cite[Lemma 3.8]{khanh-hai_almot} that $\tau(B)$ is characteristic in $A$. In order to use Lemma \ref{lemma_2.6}, we divide our situation into three cases:
	
	\bigskip
	
	\textit{Case 1. $\tau(B)$ is not abelian.}
	
	\bigskip
	
	Since $\tau(B)$ is characteristic in $B$ and $B$ is normal in $M$, we conclude that $\tau(B)$ is normal in $M$. By virtue of Lemma \ref{lemma_2.5}, we have  $F(\tau(B))=D$. In addition, as $\tau(B)$ is solvable and periodic, it is actually a locally finite group (\cite[Lemma 2.12]{khanh-hai_almot}). It follows that $D=F(\tau(B))=F[\tau(B)] $ is a locally finite division ring. Since $M$ is solvable-by-finite, it contains no non-cyclic free subgroups. With reference to \cite[Theorem 3.1]{hai-khanh}, we deduce that $[D:F]<\infty$ and $M$ is abelian-by-finite.
	
	\bigskip	
	
	\textit{Case 2. $\tau(B)$ is abelian and contains an element $x$ of order $4$ not in the center of $B=C_A(A')$.}
	
	\bigskip
	
	Since  $x\not\in Z(B)$, it does not belong to $F$. The condition $x$ is of finite order implies that the field $F(x)$ is an algebraic extension of $F$.  Note that $\left\langle x\right\rangle$ is indeed a $2$-primary component of $\tau(B)$ (see \cite[Theorem 1.1, p.132]{wehrfritz_07}); thus, it is a characteristic subgroup of $\tau(B)$. Consequently, $\left\langle x\right\rangle$ is a normal subgroup of $M$. Therefore, all elements of the set $x^M=\{m^{-1}xm\vert m\in M\}\subseteq F(x)$ have the same minimal polynomial over $F$. As a result, $x$ is an $FC$-element and so $[M:C_M(x)]<\infty$. Now, if we set $C=\core_M(C_M(x))$, then $C$ is a normal subgroup of finite index in $ M$. In view of Lemma \ref{lemma_2.5}, either $F(C)=D$ or $C$ is abelian. The first case cannot occur since it implies that $x\in F$, which is impossible. Therefore $C$ is abelian and, in consequence, $M$ is abelian-by-finite.  By the same arguments used in the first paragraph, we conclude that $D$ is centrally finite.
	
	\bigskip	
	
	\textit{Case 3. $H=AH_1$.}
	
	\bigskip
	
	The fact $A'\subseteq H_1\cap A$ implies that $H/H_1\cong A/A\cap H_1$ is abelian and so $H'\subseteq H_1$. Since $H_1$ is abelian, it follows that $H'$ is abelian too. Because $M\subseteq H$, we know that $M'$ is also abelian. In other words, $M$ is a metabelian group; hence, the conclusions follow from \cite[Proposition 3.7]{khanh-hai_almot}.
\end{Proof}

\begin{lemma}[{\cite[4.5.1]{shirvani-wehrfritz}}]\label{lemma_2.8}
	Let $D$ be a division ring that is not a locally finite field, and $n\geq 2$ an integer. If $N$ is a non-central normal subgroup of $\GL_n(D)$, then $N$ contains a non-cyclic free subgroup.
\end{lemma}

\begin{lemma}\label{lemma_2.9}
	Let $D$ be a division ring with center $F$ containing at least four elements, $G$ a normal subgroup of $\GL_n(D)$ with $n\geq2$.  Assume that $M$ is a maximal subgroup of $G$. If $A$ is an $F$-subalgebra of $\M_n(D)$ normalized by $M$, then either $A^*\cap G\subseteq M$ or $A=\M_n(D)$ provided $A^*\cap G\not\subseteq F$.
\end{lemma}

\begin{Proof}
	Since $M$ normalizes $A$, we have $M\subseteq N_{\GL_n(D)}(A^*)\cap G \subseteq G$. Since $M$ is maximal in $G$, we have either $M = N_{\GL_n(D)}(A^*)\cap G$ or $N_{\GL_n(D)}(A^*)\cap G = G$. If the first case occurs, then $G\cap A^*\subseteq M$. Now, suppose that $N_G(A^*) = G$ and that $A^*\cap G\not\subseteq F$. It is clear for this case that $A^*\cap G$ is a non-central normal subgroup of $\GL_n(D)$, from which it follows that $\SL_n(D)\subseteq A^*$ (\cite[Theorem 11]{mah98}). The consequence of this fact is that $A$ contains the subring $F[\SL_n(D)]$, which is normalized by $\GL_n(D)$. According to Cartan-Brauer-Hua Theorem for the matrix ring (see e.g. \cite[Theorem D]{dorbidi2011}), we obtain that $A=\M_n(D)$.
\end{Proof}

\begin{lemma}\label{lemma_2.10}
	Let $D$ be an infinite division ring with center $F$, and $n\geq 1$. If  $\SL_n(D)$ is (locally solvable)-by-finite, then $n=1$ and $D=F$.
\end{lemma}
\begin{proof}
	For a proof by contradiction, we assume that $n>1$. Set $S=\SL_n(D)$. From our hypthesis, $S$ contains a locally solvable normal subgroup $G$ such that $S/G$ is finite. If $G$ is not contained in $F$, then $G=S$  (\cite[Theorem 11]{mah98}), and so $S$ is locally solvable.  By Lemma \ref{lemma_2.8}, we deduce that $D$ is a field, which says that $S$ is indeed a solvable linear group; recall that every locally solvable linear group is solvable. But this leads at once to the contradiction that $S$ is a perfect group. We may therefore assume that $G\subseteq F$, from which it follows that $S/S\cap F^*$ is finite. A similar argument permits us to conclude that $F$ is again a field. Then, we also have a contradiction since it is well-known that for an infinite field $F$, the special linear group $\SL_n(F)$ contains no proper subgroups of finite index. These contradictions imply that $n=1$ and, in consequence, we have $D'$ is (locally solvable)-by-finite. With reference to Lemma \ref{lemma_2.4}, we conclude that $D'$ is contained in $F$. As a result, $D^*$ is solvable, yielding that $D=F$. The proof is completed.
\end{proof}
\begin{lemma}\label{lemma_2.11}
	Let $D$  be an infinite division ring with center $F$, and $G$ a normal subgroup of $\GL_n(D)$. Assume that $M$ is a non-abelian maximal subgroup of $G$, and that $N$ is a subnormal subgroup of $M$. If either
	\begin{enumerate}[font=\normalfont]
		\item $M$ is (locally solvable)-by-finite, $D$ is non-commutative, and $n\geq 2$ , or
		\item  $M$ is solvable, $D$ is a field, and $n\geq 5$,
	\end{enumerate}
then the followings hold:
	\begin{enumerate}[label=\textnormal{*}, font=\normalfont]
	\item[(i)] $M$ is primitive,
	\item[(ii)] $C_{\M_n(D)}(M)$ is a field,
	\item[(iii)] either $N$ is abelian or $F[N]$ is a prime and Goldie ring whose the classical right quotient ring is coincided with $\M_n(D)$,
	\item[(iv)]  $C_{\M_n(D)}(N)$ is a simple  artinian ring. 
\end{enumerate}
\end{lemma}

\begin{proof}
	We begin by proving that  $M$ is irreducible. Otherwise, it follows from \cite[1.1.1]{shirvani-wehrfritz} that there exist a matrix $P\in\GL_n(D)$ and some integer $0<m<n$ such that $PM{P^{-1}}\subseteq H$, where $$H=\left({\begin{array}{*{20}{c}}{\GL_m(D)}&*\\0&{\GL_{n - m}(D)}\end{array}}\right) \cap G.$$ The normality of $G$ and the maximality of $M$ imply that $PMP^{-1}$ is a maximal subgroup of $G$ contained in $H$, which yields either $H=G$ or $H=PMP^{-1}$. We observe that in the former event, we would have  $\SL_n(D)\subseteq H$, which  is impossible since $\I_n+\E_{n1}$ belongs to $\SL_n(D)$ but it is not an element of $H$ (here $\E_{n1}$ is the matrix with $1$ in the position $(n,1)$ and $0$ everywhere else). Therefore, we may assume that $H=PMP^{-1}$, which is conjugate to $M$.  It follows that $H$ is (locally solvable)-by-finite. Consequently, the group $$\left({\begin{array}{*{20}{c}}{\SL_m(D)}& 0\\0&{\SL_{n-m}(D)}\end{array}}\right) \subseteq H,$$ which clearly contains a copy of $\SL_m(D)$ and of $\SL_{n-m}(D)$, is  (locally solvable)-by-(finite) too. By virtue of Lemma \ref{lemma_2.10}, we must have  $m=n-m=1$ and so $n=2$ and $D=F$, a contradiction. As a result, the group $M$ is irreducible, proving our claim.
	
	As $M$ is irreducible, it follows from \cite[Lemma 8]{akbari} that $F_1=C_{\M_n(D)}(M)$ is a division ring. Take an element $x\in F_1'\subseteq \SL_n(D)$, by \cite[Theorem 11]{mah98}, we have  $x\in G$. The maximality of $M$ in $G$ again says that either $\left\langle M,x\right\rangle \cap G =M$ or $G\subseteq \left\langle M,x\right\rangle$. If the first case occurs, then $\left\langle M,x\right\rangle=M$, or $x\in M$. It follows that $x\in Z(M)$. In the latter case, we have $F[G]=F[\left\langle M,x\right\rangle]=\M_n(D)$ by the Cartan-Brauer-Hua Theorem for the matrix ring and so $x\in F$. Thus, in either case we have $x\in F^*Z(M)$, from which we conclude that $F_1'$ is abelian. A consequence of this fact is $F_1^*$ is solvable, and so $F_1$ is indeed a field. The assertion (ii) is now established.
	
	For a proof of (i),  we assume by contradiction that $M$ is imprimitive. Then, the proof of \cite[Lemma 2.5]{hai-khanh} says that $M$ contains a copy of $\SL_r(D)\wr S_k$, the wreath product of $\SL_r(D)$ and the symmetric group $S_k$ for some $r > 1$ and $n=rk$. Since $M$ is (locally solvable)-by-finite, so is $\SL_r(D)$. By Lemma \ref{lemma_2.10}, we have $r=1$, $k=n$ and $D$ is a field, which contradicts to the assumption of (1). Therefore $M$ is primitive under assumption (1). If $n\geq 5$ and $M$ is solvable, then the fact $n=k$ implies that $S_n$ is solvable, a contradiction. This contradiction indicates that $M$ is primitive, and (i) is proved.
	
	To prove (iii), we note that $R$ is in fact Goldie (\cite[Corollary 24]{wehrfritz_89}) and prime. Therefore, the classical right quotient ring, say $Q$, of $R$ must be simple artinian (\cite[Theorem 6.18]{goodearl-warfield}), which is embedded in $\M_n(D) $ by \cite[5.7.8]{shirvani-wehrfritz}.
	Therefore, the exist a division $F_1$-algebra $E$ and an integer $m\geq 1$ such that $Q\cong \M_m(E)$. Since $M$ normalizes $R$, it also normalizes $Q$. It follows from Lemma \ref{lemma_2.9} that either $Q^*\cap G\unlhd  M$ or $Q=\M_n(D)$. The first case implies that $Q^*\cap G$ is a subnormal subgroup of $Q^*$ contained in $M$. As a result, the subgroups $N\subseteq Q^*\cap  G$ is contained  in $Z(Q)$ by  Lemma \ref{lemma_2.4} and Lemma \ref{lemma_2.8}. In other words, $N$ is abelian, and (iii) is proved.
	
	The confirmation of the final assertion (iv) follows from the same argument used in the proof of \cite[Proposition 3.3]{dorbidi2011}. 
\end{proof}

\begin{lemma}[{\cite[Theorem 2]{lanski_81}}]\label{lemma_2.12}
	Let $R$ be a prime ring with identity, $Z=Z(R)$ the center of $R$ containing at least five elements, and $\overline{U}$ the $Z$-subalgebra of $R$ generated by $R^*$. Assume $\overline{U}$ contains a nonzero ideal of $R$. If $R^*$ has a solvable normal subgroup which is not central, then $R$ is a domain.
\end{lemma}

\begin{lemma}\label{lemma_2.13}
	Let $A$, $B$, $C$ be groups such that $A\unlhd B\unlhd C$ and that $A$ is a solvable subgroup of finite index in $B$. Then $A$ is contained in a solvable subgroup of $B$ and such a subgroup is normal in $C$.
\end{lemma}

\begin{proof}
	For any $x\in C$, it is clear that $x^{-1}Ax$ is a solvable normal subgroup of $B$. If we set $H=\left\langle x^{-1}Ax\right\rangle $ where $x$ runs over $C$, then $A\unlhd H \subseteq B$ and $H \unlhd C$. Since $[B:A]$ is finite, we conclude that $[H:A]$ is finite. Let $\{h_1, \dots, h_n\}$ be a transversal of $A$ in $H$. Since $H=\left\langle x^{-1}Ax\right\rangle $, for each $1\leq i\leq n$, the element $h_i$ may be expressed in the form
	$$h_i=x_{i_1}^{-1}a_{i_1}x_{i_1}\cdots x_{i_{k_i}}^{-1}a_{i_{k_i}}x_{i_{k_i}},$$ where $x_{i_j}\in C$ and $a_{i_j}\in A$. It is clear that $$H=\left\langle x_{i_j}^{-1}Ax_{i_j}: 1\leq i\leq n, 1\leq j\leq k_i\right\rangle.  $$ Since there are only finitely many $x_{i_j}$'s, it follows that $H$ is a solvable normal subgroup of $C$ containing $A$.
\end{proof}

\begin{theorem}\label{theorem_2.14}
	Let $D$ be a non-commutative division ring with center $F$ which contains at least five elements, and $G$ a normal subgroup of $\GL_n(D)$ with $n\geq2$. If $M$ is a solvable-by-finite maximal subgroup of $G$ such that $F[M]\ne\M_n(D)$, then $M$ is abelian.
\end{theorem}
\begin{proof}
	Let $N$ be a solvable normal subgroup of finite index of $M$. We shall show first that $N$ is abelian.  For this purpose, we set $R=F[N]$ and $Q$ to be the classical right quotient ring of $R$. Then, Lemma \ref{lemma_2.11}(iii) ensures that $R$ is prime and Goldie, and that either $N$ is abelian or $Q=\M_n(D)$. If the first case occurs, then we are done. Now, we shall show that the latter case cannot happen by contradiction. So, assume that $Q=\M_n(D)$ and that $N$ is non-abelian. In view of Lemma \ref{lemma_2.9}, we have $G\cap R^*\subseteq M$, from which it follows  that  $G\cap R^*$ is a normal subgroup of $R^*$ contained in $M$. Now, we have $N \unlhd G\cap R^* \unlhd R^*$ and $[G\cap R^*:N]<\infty$. With reference to Lemma \ref{lemma_2.13}, we conclude that $N$ is contained in a solvable normal subgroup, say $H$, of $R^*$. Since $Z(R)$ contains $F$, it has at least five elements. Additionally, the fact $N\subseteq H$ assures us that $H$ is not contained in $Z(R)$. Therefore, we may apply Lemma \ref{lemma_2.12} to obtain that $R$ is a domain. Now, $R$ is both a domain and Goldie, it is actually an Ore domain, and so $Q=\M_n(D)$ is a division ring. But this leads to a contradiction that $n>1$. Therefore, we have $N$ is abelian, and so $M$ is abelian-by-finite. 
	
	Next, we assert that $M$ is indeed abelian. Again, Lemma \ref{lemma_2.11}(iii) shows that $S=F[M]$ is a prime ring. Because $M$ is abelian-by-finite, we may apply \cite[Lemma 11, p.176]{passman_77} to conclude that the group ring $FM$ is a PI-ring. Thus, as a hommomorphic image of $FM$, the ring $S$ is also a PI-ring. Since $M$ normalizes $S$, by Lemma \ref{lemma_2.9}, we deduce that $S^*\cap G\subseteq F$ or $S^*\cap G\subseteq M$. The first case yields that $M$ is abelian, and we are done. The latter illustrates that $M= S^*\cap G$. In view of Lemma \ref{lemma_2.11}(i), we obtain that $F_1=C_{\M_n(D)}(M)$ is a field. There are two possible cases:
	
	\bigskip 
	
	\textit{Case 1.} $F_1\subseteq M$. In this case, the field $F_1$ is the center of the prime ring $S$. In view of \cite[Corollary 1.6.28]{rowen}, we conclude that $S$ is a simple ring. Now $S$ is both simple and PI-ring, so it is a simple artinian ring by Kaplansky's theorem. Therefore, $S$ is coincided with its classical right quotient $Q$ and thus we may apply Lemma \ref{lemma_2.11}(iii) to get that $M$ is abelian.
	
	\bigskip
	
	\textit{Case 2.} $F_1\not\subseteq M$. Set $M_1=F_1^*M$ and $N_1=F_1^*N$. If $M_1=N_1$, then $M_1$ contains $\SL_n(D)$, which is impossible since $\SL_n(D)$ cannot be abelian-by-finite. If $M_1\ne N_1$, then $M_1$ is a maximal subgroup of $N_1$. By the same way, we conclude that $M_1$ is abelian, and so is $M$. This completes the proof of the theorem.
\end{proof}

\begin{lemma}\label{lemm_2.15}
	Let $D$ be a division ring with center $F$, and $M$ a subgroup of $\GL_n(D)$ with $n\geq 1$. If $M/M\cap F^*$ is a locally finite group, then $F[M]$ is a locally finite dimensional vector space over $F$.
\end{lemma}

\begin{Proof}  
	Pick a finite subset $S=\{x_1,x_2,\dots,x_k\}$ of $ F[M]$. Then, for each $1\leq i\leq k$, we can find elements $f_{i_1}, f_{i_2},\dots,f_{i_s}$ in $F$ and $m_{i_1}, m_{i_2},\dots,m_{i_s}$ in $M$ such that
	$$x_i=f_{i_1}m_{i_1}+f_{i_2}m_{i_2}+\cdots+f_{i_s}m_{i_s}.$$
	Let $G$ be the subgroup of $M$ generated by the $m_{i_j}$'s.
	We know that $M/M\cap F^*\cong MF^*/F^*$ and so $MF^*/F^*$ is locally finite, from which it may be concluded that $GF^*/F^*$ is finite.  If $\{y_1,y_2,\dots,y_t\}$ is a transversal of $F^*$ in $GF^*$, then 
	$$R=Fy_1+Fy_2+\cdots+Fy_t$$ 
	forms a ring containing $S$. The displayed relation also means $R$ is finite dimensional over $F$. Since $S$ is chosen to be arbitrary, our result certainly follows. 
\end{Proof}

\begin{lemma}\label{lemma_2.16}
	Let $R$ be a ring, and $G$ a subgroup of $R^*$. Assume that $F$ is a central subfield of $R$ and that $A$ is a maximal abelian subgroup of $G$ such that $K=F[A]$ is normalized by $G$. Then $F[G]$ is a crossed product of $K$ by $G/A$. In addition, if $K$ is a field, then it is a maximal subfield of $R$.
\end{lemma}

\begin{Proof} 
	Since $K$ is normalized by $G$, it follows that $F[G]=\sum\nolimits_{g \in T}{Kg}$ for every transversal $T$ of $A$ in $G$. Thus, to establish that $F[G]$ is a crossed product of $K$ by $G/A$, it suffices to prove that every finite subset $\{g_1,g_2,\dots,g_n\}$ of $T$ is  linearly independent over $K$. For a purpose of contradiction, we assume that there exists such a non-trivial relation 
	$$k_1g_1+k_2g_2+\cdots+k_ng_n=0.$$
	Clearly, we can suppose that all the $k_i$'s are non-zero and that $n$ is minimal. The case where $n=1$ is obviously trival and so we suppose that $n>1$. As the cosets $Ag_1$ and $Ag_2$ are disjoint, we know that $g_1^{-1}g_2\not\in A=C_G(A)$. Therefore, there exists an element $x\in A$ for which $g_1^{-1}g_2x\ne xg_1^{-1}g_2$. For each $1\leq i\leq n$, if we set $x_i=g_ixg_i^{-1}$, then $x_1\ne x_2$. Since $G$ normalizes $K$, it follows $x_i\in K$ for all $1\leq i\leq n$. Now, we have 
	$$(k_1g_1+\cdots+k_ng_n)x-x_1(k_1g_1+\cdots+k_ng_n)= 0.$$
	By definition of the $x_i$'s, we deduce that $x_ig_i=g_ix$, and so $x$, $x_i\in K$ for all $i$. By the fact that $K=F[A]$ is commutative, the last equality reveals
	$$\left( {{x_2} - {x_1}} \right){k_2}{g_2} +  \cdots  + \left( {{x_n} - {x_1}} \right){k_n}{g_n} = 0,$$
	which is a non-trivial relation (since $x_1\ne x_2$) with less than $n$ summands, contrasting to the minimality of $n$. As a result, we obtain the desired fact that $T$ is linearly independent over $K$.	
	
	Regarding the last assertion of our lemma,  we assume that $R=F[G]$ and that $K$ is a field. If we set $L=C_R(K)$, then every element $y\in L$ may be written in the form
	$$y=l_1m_1+l_2m_2+\cdots+l_tm_t,$$ 
	where $l_1,l_2,\dots,l_t\in K$ and $m_1,m_2,\dots,m_t\in T$. Take an arbitrary element  $a\in A$, by the normality of $A$ in $M$, there exist $a_i\in A$ such that $m_ia=a_im_i$ for all $1\leq i\leq t$. Since $ya=ay$, it follows that
	$$ (l_1a_1-l_1a)m_1+(l_2a_2-l_2a)m_2+\cdots+(l_ta_t-l_ta)m_t=0.$$ 
	As $\{m_1,m_2,\dots,m_t\}$ is linearly independent over $K$, the outcome is that $a=a_1=\cdots=a_t$. Consequently, $m_ia=am_i$ for all $a\in A$; thus, $m_i\in C_M(A)=A$ for all $1\leq i\leq t$. The consequence of this fact is that $y\in K$, yielding $L=K$. This implies that $K$ is a maximal subfield of $R$, and our proof is now completed.
\end{Proof}

\begin{lemma}[{\cite[3.2]{wehrfritz_91}}]\label{lemma_2.17}
	Let $R$ be a ring, $J$ a subring of $R$, and $H\leq K$ subgroups of the group of units of $R$ normalizing $J$ such that $R$ is the ring of right quotients of $J[H]\leq R$ and $J[K]$ is a crossed product of $J[B]$ by $K/B$ for some normal subgroup $B$ of $K$. Then $K=HB$.
\end{lemma}

\begin{theorem} \label{theorem_2.18}
	Let $D$ be non-commutative division ring with center $F$ which contains at least four elements, and $G$ a normal subgroup of $\GL_n(D)$ with $n\geq2$. If $M$ is a non-abelian solvable-by-finite maximal subgroup of $G$ such that $F[M]=\M_n(D)$, then $[D:F]<\infty$.
\end{theorem}

\begin{proof} 
	First, we observe that $M$ is abelian-by-locally finite (\cite[Theorem 1]{wehrfritz84}). As a result, there exists in $M$ a maximal subgroup $A$ with respect to the property: $A$ is an abelian normal subgroup of $M$ and that $M/A$ is locally finite. In view of  \cite[1.2.12]{shirvani-wehrfritz}, we conclude that $F[A]$ is a semisimple artinian ring; thus, the Wedderburn-Artin Theorem implies that
	$$ F[A] \cong \M_{n_1}(D_1)\times \M_{n_2}(D_2)\cdots\times \M_{n_s}(D_s),$$
	where $D_i$ are division $F$-algebras, $1\leq i\leq s$. Since $F[A]$ is abelian, it follows that the $n_i$'s are equal to 1 and $D_i$'s are fields that contain $F$. Consequently,  
	$$F[A]\cong K_1\times K_2\cdots\times K_s.$$
	With reference to Lemma \ref{lemma_2.11}(iii), we conclude that $F[A]$ is an integral domain and so $s=1$. It follows that $K:=F[A]$ is a subfield of $\M_n(D)$ containing $F$.

	If we set $L=C_{\M_n(D)}(K)$, then by Lemma \ref{lemma_2.11}(iv), one has $L\cong \M_{m}(E)$ for some division $F$-algebra $E$ and some integer $m\geq 1$. Since $M$ normalizes $K$, it also normalizes $L$, and hence either $L^*\cap G\subseteq F$, or $L=\M_n(D)$ or $L^*\cap G\subseteq M$ by Lemma \ref{lemma_2.9}. The first case implies that $\M_n(D)=K$, which contradicts the fact that $n>1$. If the second case occurs, then $A\subseteq F$, from which it follows that $M/M\cap F^*$ is locally finite. In view of Lemma \ref{lemm_2.15}, one has $D$ is a locally finite division ring. Since $M$ contains no non-cyclic free subgroups, by \cite[Theorem 3.1]{hai-khanh}, we conclude that $M$ is abelian-by-finite and $[D:F]<\infty$; we are done.  Now, we consider the third case $L^*\cap G \subseteq M$, which yields that $L^*\cap G$ is a solvable-by-finite normal subgroup of $\GL_{m}(E)$. It follows by Lemma \ref{lemma_2.4} and Lemma \ref{lemma_2.8} that $L^*\cap G\subseteq Z(E)$. In any case, we obtained that $L^*\cap G$ is an abelian normal subgroup of $M$ and $M/L^*\cap G$ is locally finite. By the maximality of $A$ in $M$, it follows $A=L^*\cap G=L^*\cap M=C_M(A)$. Hence, $A$ is actually a maximal abelian subgroup of $M$. Therefore, we may apply Lemma \ref{lemma_2.16} to conclude that $F[M]=\M_n(D)$ is a crossed product of $K$ by $M/A$, and that $K$ is a maximal subfield of $\M_n(D)$.
	
	If we set $L=C_{\M_n(D)}(K)$, then by Lemma \ref{lemma_2.11}(iv), one has $L\cong \M_{m}(E)$ for some division $F$-algebra $E$ and some integer $m\geq 1$. Since $M$ normalizes $K$, it also normalizes $L$; hence, either $L^*\cap G\subseteq F$, or $L=\M_n(D)$ or $L^*\cap G\subseteq M$ (Lemma \ref{lemma_2.9}). The first case implies that $\M_n(D)=K$, which is impossible since $n>1$. If the second case occurs, then $A\subseteq F$, from which it follows that $M/M\cap F^*$ is locally finite. According to \ref{lemm_2.15}, we deduce that $D$ is a locally finite division ring. Now, we can use \cite[Theorem 3.1]{hai-khanh} to conclude that $[D:F]<\infty$. It remains only to consider the third case $L^*\cap G \subseteq M$, from which we have $L^*\cap G$ is a solvable-by-finite normal subgroup of $\GL_{m}(E)$. By applying both Lemmas \ref{lemma_2.4} and \ref{lemma_2.8} to this situation, we obtain that $L^*\cap G\subseteq Z(E)$. Thus, $L^*\cap G$ is an abelian normal subgroup of $M$ and $M/L^*\cap G$ is locally finite. The maximality of $A$ in $M$ yields  $A=L^*\cap G=L^*\cap M=C_M(A)$. These equalities implies that $A$ is actually a maximal abelian subgroup of $M$. In view of Lemma \ref{lemma_2.16}, we conclude that $F[M]=\M_n(D)$ is a crossed product of $K$ by $M/A$, and that $K$ is a maximal subfield of $\M_n(D)$.
	
	Next, we prove that $M/A$ is simple. Suppose that $B$ is an arbitrary normal subgroup of $M$ properly containing $A$. Note that by the maximality of $A$ in $M$, we conclude that $N$ is non-abelian. If we set $R=F[B]$ and $Q$ to be its quotient ring,  Lemma \ref{lemma_2.11}(iii) says that $Q=\M_n(D)$. Now, we may apply Lemma \ref{lemma_2.17} to conclude that $M=AB=B$; recall that $A\subseteq B$. This yields that $M/A$ is simple. 
	
	Our next step is to prove that $M/A$ is simple. To do so, assume that $B$ is an arbitrary normal subgroup of $M$ properly containing $A$. Note that by the maximality of $A$ in $M$, we may assume further that $N$ is non-abelian. If we set $R=F[B]$ and $Q$ to be its quotient ring,  then Lemma \ref{lemma_2.11}(iii) says that $Q=\M_n(D)$. Now, we may apply Lemma \ref{lemma_2.17} to conclude that $M=AB=B$; recall that $A\subseteq B$. It follows that $M/A$ is simple.
	
	Since $M$ is solvable-by-finite, it contains a solvable normal subgroup $N$ such that $M/N$ is finite. Because $AN$ is a normal subgroup of $M$, the simplicity of $M/A$ shows that $AN=M$ or $AN=A$. The first case implies $M$ is solvable. Now, $M/A$ is simple and solvable, one has $M/A\cong \Z_p$ for some prime number $p$. By Lemma~ \ref{lemma_2.16}, it follows $\dim_K\M_n(D)=|M/A|=p$, which forces $n=1$, a contradiction. Thus, we have $AN=A$, from which it follows that $[M:A]<\infty$. Again by Lemma \ref{lemma_2.16}, one has $[\M_n(D):K]=|M/A|<\infty$, and hence $[D:F]<\infty$.
	
	Since $M$ is solvable-by-finite, it contains a solvable normal subgroup $N$ such that $M/N$ is finite. As $AN$ is also a normal subgroup of $M$, the simplicity of $M/A$ shows that $AN=M$ or $AN=A$. The first case implies $M$ is solvable, which also says that $M/A$ is solvable. Now, $M/A$ is simple and solvable, one has $M/A\cong \Z_p$ for some prime number $p$. By Lemma \ref{lemma_2.16}, it follows that $\dim_K\M_n(D)=|M/A|=p$, which forces $n=1$, an obvious contradiction. As a result, we must have $AN=A$, from which it follows that $[M:A]<\infty$. Again by Lemma \ref{lemma_2.16}, one has $[\M_n(D):K]=|M/A|<\infty$ and so $[D:F]<\infty$.
\end{proof}

\begin{corollary}\label{corollary_2.19}
	Let $D$ be a non-commutative division ring with center $F$ which contains at least five elements, and $G$ a normal subgroups of $\GL_n(D)$ with $n\geq2$. If $M$ is a solvable maximal subgroup of $G$, then $M$ is abelian.
\end{corollary}
\begin{proof}
	If $F[M]\ne \M_n(D)$, then the result follows from Theorem \ref{theorem_2.14}. In the case $F[M]=\M_n(D)$, if $M$ is non-abelian, then  the last paragraph of the proof of the above theorem says that $n=1$, a contradiction.
\end{proof}

Here now is the proof of the main theorem of this section.

\bigskip

\noindent{\it \textbf{Proof of Theorem~\ref{theorem_1.1}}}. Combining three Theorems \ref{theorem_2.7}, \ref{theorem_2.14} and \ref{theorem_2.18}, we get $[D:F]<\infty$. By hypothesis, we conclude that $M$ contains no non-cyclic free subgroups. Therefore, most of conclusions follow from \cite[Theorem 3.1]{hai-khanh}. From the maximality of $K$, we deduce that $C_{\M_n(D)}(K)=K$. By Centralizer Theorem (\cite[(vii), p.42] {draxl}), one has $[K:F]^2=[\M_n(D):F]=n^2[D:F]$. It follows that $|M/K^*\cap G|=|\Gal(K/F)|=[K:F]=n\sqrt{[D:F]}$.   $\square$

\bigskip

Other results concerning solvable-by-finite subgroups of $\GL_n(D)$, where $D$ is a centrally finite division ring and $n\geq 1$, were nicely obtained by Wehrfritz in \cite{wehrfritz_07}. 
In fact, he proved that if $M$ is a solvable-by-finite subgroup of $\GL_n(D)$, then it contains an abelian normal subgroup of index dividing $b(n)[D:F]^n$, where $b(n)$
is an integer valued function that depends only on $n$. In view of Theorem~\ref{theorem_1.1}, if $M$ is supposed further to be a maximal subgroup of an almost subnormal subgroup of $\GL_n(D)$, then $M$ possesses  an abelian normal subgroup of very explicit index.

Theorem~\ref{theorem_1.1} also gives some interesting corollaries having close relation to the results obtained in \cite{dorbidi2011}, \cite{hazrat}, \cite{nassab14}, and \cite{wehrfritz_07}. 
More precisely, the authors of  \cite{hazrat} asked whether a division $D$ is a crossed product if the multiplicative group $D^*$ contains an absolutely irreducible solvable-by-finite subgroup $M$ (Question 2.5).  
By definition, a centrally finite division $D$  is called a \textit{crossed product} if it contains a maximal subfield that is a Galois extension over the center of $D$. The following corollary, which follows immediately from  Theorem \ref{theorem_1.1}, shows that the question has a positive answer if $M$ is a non-abelian solvable-by-finite maximal subgroup of an almost subnormal subgroup of $D^*$.

\begin{corollary}
	Let $D$ be a division ring, and $G$ an almost subnormal subgroup of $D^*$. If $M$ is a non-abelian solvable-by-finite maximal subgroup of $G$, then $D$ is a crossed product.
\end{corollary}

Polycyclic-by-finite maximal subgroups of $\GL_n(D)$  have already been studied in \cite{nassab14}. One of the main results of \cite{nassab14} states that $\GL_n(D)$ contains no polycyclic-by-finite maximal subgroups if $n=1$ or the center of $D$ contains at least five elements (\cite[Theorem B]{nassab14}). In the next corollary, we extend this result to polycyclic-by-finite maximal subgroups of an almost subnormal subgroup of $\GL_n(D)$.

\begin{corollary}
	Let $D$, $F$, $G$ as in Theorem \ref{theorem_1.1}. If $M$ is finitely generated solvable-by-finite maximal subgroup of $G$, then $M$ is abelian. In particular, if $M$ polycyclic-by-finite, then it is abelian.
\end{corollary} 
\begin{proof}
	With reference to Theorem \ref{theorem_1.1}, we have $[D:F]<\infty$. But then $M$ cannot be finitely generated in view of \cite[Corollary 3]{mah2000}. The rest of our corollary follows immediately from the fact that every polycyclic-by-finite group is finitely generated.
\end{proof}

\section{Locally solvable maximal subgroups}

In this section, we study locally solvable maximal subgroups of an almost subnormal subgroup of $\GL_n(D)$, with $n\geq 1$. We note that in the case $n=1$, the following results were obtained in \cite{khanh-hai_almot}. 
\begin{theorem}[{\cite[Theorem 3.6]{khanh-hai_almot}}]\label{theorem_3.1}
	Let $D$ be a division ring with center $F$, and $G$ an almost subnormal subgroup of $D^*$. If $M$ is a locally nilpotent maximal subgroup of $G$,  then $M$ is abelian.
\end{theorem}

\begin{theorem}[{\cite[Theorem 3.7]{khanh-hai_almot}}]\label{theorem_3.2}
	Let $D$ be a division ring with center $F$, and $G$ an almost subnormal subgroup of $D^*$. If $M$ is a non-abelian locally solvable maximal subgroup of $G$, then the following hold:
	\begin{enumerate}[font=\normalfont]
		\item There exists a maximal subfield $K$ of $D$ such that $K/F$ is a finite Galois extension with $\mathrm{Gal}(K/F)\cong M/K^*\cap G\cong \mathbb{Z}_p$ and $[D:F]=p^2$, for  some prime number $p$. 
		\item The subgroup $K^*\cap G$ is the $FC$-center. Also, $K^*\cap G$ is the Hirsch-Plotkin radical of $M$. For any $x\in M\setminus K$, we have $x^p\in F$ and $D=F[M]=\bigoplus_{i=1}^pKx^i$.
	\end{enumerate}
\end{theorem}

For any group $G$, we denote by $\tau(G)$ the unique maximal locally finite normal subgroup of $G$, $\eta(G)$ the Hirsch-Plotkin radical of $G$, and $\alpha(G)$ and $\beta(G)$ the two subgroups of $G$ which are defined by 
$$\beta(G)/\tau(G)=\eta(G/\tau(G)) \;\;\;\;\;\;\hbox{and} \;\;\;\;\;\; \alpha(G)/\tau(G)=Z(\beta(G)/\tau(G)).$$

\begin{lemma}[{\cite{wehrfritz_90_criteriaII}}]\label{lemma_3.3}
	Let $G$ be a locally solvable primitive subgroup of $\GL_n(D)$ with $n\geq 1$. Then the $F$-subalgebra $F[G]$ of the full matrix ring $\M_n(D)$ generated by $G$ is a crossed product over the (locally finite)-by-abelian normal subgroup $\alpha(G)$ of $G$.
\end{lemma}

\noindent{\it \textbf{Proof of Theorem~\ref{theorem_1.2}}}. For purposes of contradiction, assume that $M$ is not abelian. The crucial step in our proof is to point out that $[D:F]<\infty$. For, since $M$ is primitive (Lemma \ref{lemma_2.11}(i)), it follows from previous lemma that  $F[M]$ is a crossed product over the (locally finite)-by-abelian normal subgroup $\alpha(M)$ of $M$. Let $A$ be the maximal (locally finite)-by-abelian normal subgroup of $M$ containing $\alpha(M)$. If $B$ is a  normal subgroup of $M$ properly containing $A$, then by the same arguments used in the second paragraph of the proof of Theorem \ref{theorem_2.18}, we conclude that $M=B\alpha(M)=B$ and, in consequence, the group $M/A$ is simple. 

Since $M/A$  is simple and locally solvable, it is a finite group of prime degree. Let $T$ be a locally finite normal subgroup of $A$ such that $A/T$ is abelian.  Then,  Lemma \ref{lemma_2.11}(iii) implies that  $F[T]$ is prime. Because  $T$ is locally finite, we conclude that $F[T]$ is simple artinian (\cite[1.1.14]{shirvani-wehrfritz}). Thus, Lemma \ref{lemma_2.11}(iii) again says that either $F[T]=\M_n(D)$ or $T$ is abelian. If the first case occurs, then Lemma \ref{lemm_2.15} implies that $D$ is locally finite, and thus $[D:F]<\infty$ by \cite[Theorem 3.1]{hai-khanh}; we are done. If the second case occurs, then $A$ is solvable, and $M$ is thus solvable-by-finite. With reference to Theorem \ref{theorem_1.1}, we conclude that $[D:F]<\infty$.

 Set $k:=[D:F]$. By viewing $M$ as a (linear) subgroup of  $\GL_{kn}(F)$, we conclude that it is solvable.  It follows from Corollary \ref{corollary_2.19} that $M$ is abelian, and we arrive at the desired contradiction. The proof is now completed. $\square$
 
 \bigskip

The proof of Theorem~\ref{theorem_1.2} depends strongly on the assumption that $D$ is non-commutative. To deal with the commutative case, we need some different approaches. The rest of the present paper aims at considering Theorem~\ref{theorem_1.2} in the case where $D$ is a field.

\begin{remark}
	Let $D$ be a division ring and $n\geq 1$ an integer.  As mentioned in the introduction, a subgroup $G$ of $\GL_n(D)$ is called an absolutely irreducible skew linear group over $D$ if $F[G]=\M_n(D)$. For linear case, there are various equivalent definitions of this concept. To be more specific, suppose that $G$ is a subgroup of  $\GL_n(F)$ for some field $F$. Then $G$ is said to be absolutely irreducible if one of the following equivalent conditions holds: $F[G]=\M_n(F)$; $G$ is irreducible over every field extension of $F$; the centralizer of $G$ in $\M_n(F)$ is $F$ (see  \cite[p.10]{shirvani-wehrfritz}). Therefore, the following lemma is immediate.
\end{remark}

\begin{lemma}\label{lemma_3.4}
	Let $F$ be a field, $\bar{F}$ the algebraic closure of $F$, and $M$ a subgroup of $\GL_n(F)$ with $n\geq 1$. If $M$ is absolutely irreducible over $F$, then it is absolutely irreducible over $\bar{F}$.
\end{lemma}

\noindent{\it \textbf{Proof of Theorem~\ref{theorem_1.3}}}. In order to arrive at a contradiction, we assume that $M$ is non-abelian. If we set $R=F[M]$, then either $R^*\cap G\subseteq M$ or else $R=\M_n(F)$ by Lemma \ref{lemma_2.9}. If the first case occurs, by the same arguments used (for $N$) in the proof of Theorem \ref{theorem_2.14} says that $M$ is abelian, an obvious contradiction. We may therefore assume that $R=\M_n(F)$, which means that $M$ is absolutely irreducible over $F$. It follows from \cite[Lemma 3.5(ii)]{wehrfritz_infinite-linear-group} that $M$ is abelian-by-finite. Let $A$ be a maximal subgroup of $M$ with respect to the property: $A$ is an abelian normal subgroup of $M$ such that $M/A$ is finite. If $K=F[A]$, then the primitivity (see Lemma \ref{lemma_2.11}(i)) of $M$ yields that $K$ is a prime ring, and so this is actually an integral domain. Since $\dim_FK<\infty$, we conclude that $K$ is a subfield of $\M_n(F)$. If we set $L=C_{\M_n(F)}(A)$, then there exist a division $F$-algebra $E$ and an integer $m\geq 1$ such that $L\cong \M_{m}(E)$ (Lemma \ref{lemma_2.11}(iv)). It follows from Lemma \ref{lemma_2.9} that either $L=\M_n(F)$ or $L^*\cap G\subseteq M$. There are two possible cases.

\bigskip

\textit{Case 1: } $L=\M_n(F)$. It is clear that  $A\subseteq K\subseteq F$. Since we are in the case $F[M]=\M_n(F)$, we conclude that $Z(M)=M\cap F^*$. Thus, the condition $[M:A]<\infty$ and $A\subseteq F$ imply that $[M:Z(M)]<\infty$. As a result, we can take a transversal $\{x_1,\dots,x_k\}$ of $Z(M)$ in $M$. Pick an element $x\in G\backslash M$ and set $H=\left\langle x, x_1,\dots,x_k\right\rangle $. The maximality of $M$ in $G$ assures us to conclude that $G=HZ(M)$. Since $G'=(HZ(M))'=H'\subseteq G$, it follows that $H$ is normal in $G$, from which it follows that $H$ is a finitely generated subnormal subgroup of $\GL_n(F)$. By virtue of \cite{mah2000}, we have $H\subseteq F$, and so $M$ is abelian, which is a desired contradiction. 

\bigskip

\textit{Case 2: } $L^*\cap G\subseteq M$. In this case, one has $L^*\cap G$ is a solvable normal subgroup of $\GL_m(E)$, from which it follows that $L^*\cap G$ is an abelian normal subgroup of $M$. The maximality of $A$ in $M$ shows that $L^*\cap G=L^*\cap M=A$. Consequently, $A$ is indeed a maximal abelian subgroup of $M$. Therefore, we may apply Lemma \ref{lemma_2.16} to conclude that $F[M]=\M_n(F)$ is a crossed product of $K$ by $M/A$, and that $K$ is a maximal subfield of $\M_n(F)$. By the same arguments used in the second paragraph of the proof of Theorem \ref{theorem_2.18}, we conclude that $M/A$ is a simple group.

Since $M/A$ is solvable and simple, it is of prime degree, say $p$. 
Consequently, we have $n=p$ and $[K:F]=p$. Let $\bar{F}$ be the algebraic closure of $F$. In view of Lemma \ref{lemma_3.4}, we conclude that $\bar{F}[M]=\M_n(\bar{F})$. It follows that $Z(M)=M\cap \bar{F}^*$. The condition $A\subseteq K^*\subseteq \bar{F}^*$ implies  that $A\subseteq Z(M)$; hence, we deduce $A=Z(M)$ by the maximality of $A$. Also, the fact $F[M]=\M_n(F)$ yields that $Z(M)=M\cap F^*$, from which it follows that $A\subseteq F$. This being the case, we obtain that $K=F[A]=F$, contrasting to the fact that $[K:F]=p$ is a prime number. $\square$

\bigskip

\noindent

\textbf{Acknowledgements.}
The authors are profoundly grateful to the referees for their careful reading and suggestions which led a significantly improved revised version of the manuscript.


\begin{thebibliography}{}
	
		
	\bibitem{akbari} S. Akbari, R. Ebrahimian, H. Momenaei Kermani, A. Salehi Golsefidy, Maximal subgroups of $\GL_n(D)$, J. Algebra 259 (2003) 201-225.
		
	\bibitem{deo-bien-hai-19} T. T. Deo,  M. H. Bien, B. X. Hai, On division subrings normalized by almost subnormal subgroups in division rings, Period. Math. Hung. 80 (2020) 15-27.
	
	\bibitem{dorbidi2011}  H. R. Dorbidi, R. Fallah-Moghaddam, M. Mahdavi-Hezavehi, Soluble maximal subgroups in $\GL_n(D)$, J. Algebra Appl. 10(6) (2011) 1371-1382.  
	
	\bibitem{draxl}  P. Draxl, Skew Fields, London Math. Soc. Lecture Note Ser. 81, Cambridge University Press, 1983.
	
	\bibitem{ebrahimian_04}  R. Ebrahimian, Nilpotent maximal subgroups of $GL_n(D)$, J. Algebra 280 (2004) 244–248.
	
	\bibitem{moghaddam} R. Fallah-Moghaddam, Maximal subgroups of $\SL_n(D)$, J. Algebra 531 (2019) 70-82.
	
	\bibitem{goodearl-warfield} K. R. Goodearl, R. B. Warfield, An Introduction to Noncommutative Noetherian Rings, London Math. Soc. 61, Cambridge University Press, Cambridge, 2004.
	
	\bibitem{hai-minh-bien-chua} B. X. Hai,  B-M. Bui-Xuan, M. H. Bien, and L. V. Chua, Intersection graphs of almost subnormal subgroups in general skew linear groups, arXiv:2002.06522v1.
	
	\bibitem{hai-khanh}  B. X. Hai, H. V. Khanh, Free subgroups in maximal subgroups of skew linear groups, Internat. J. Algebra Comput. Internat. J. Algebra Comput. 29(3) (2019) 603-614.
	
	\bibitem{Hartley_89} B. Hartley, Free groups in normal subgroups of unit groups and arithmetic groups, Contemp. Math. 93 (1989) 173-177.
	
		
	\bibitem{Haz-Wad} R. Hazrat and  A. R. Wadsworth, On maximal subgroups of the multiplicative group of a division algebra, Journal of Algebra, 322(7) (2009) 2528-2543.
	
	\bibitem{hazrat}  R. Hazrat, M. Mahdavi-Hezavehi and M. Motiee, Multiplicative groups of division rings, Mathematical Proceedings of the Royal Irish Academy 114A (2014) 37-114.
		
	\bibitem{her} I. N. Herstein, Multiplicative commutators in division rings, Israel J. Math. 31 (1978) 180-188.
	
	\bibitem{khanh-hai} H. V. Khanh, B. X. Hai, A note on solvable maximal subgroups in subnormal subgroups of $\GL_n(D)$, arXiv:1809.00356v2 [math.RA].
			
	\bibitem{khanh-hai_almot} H. V. Khanh, B. X. Hai, On almost subnormal subgroups in division ring, arXiv:1908.04925v2 [math.RA].
	
	\bibitem{tylam_FC} T. Y. Lam, A first course in noncommutative rings, 2nd ed., Graduate Texts in Mathematics, vol. 131, Springer-Verlag, New York, 2001.
	
	\bibitem{lanski_81} C. Lanski, Solvable subgroups in prime rings,  Proc. Amer. Math. Soc. 82 (1981) 533-537. 
	
	\bibitem{mah98} M. Mahdavi-Hezavehi and S. Akbari, Some special subgroups of $\GL_n(D)$, Algebra Colloq. 5(4) (1998) 361-370.
	
	\bibitem{mah2000} M. Mahdavi-Hezavehi, M. G. Mahmudi, and S. Yasamin, Finitely generated subnormal subgroups of $\GL_n(D)$ are central, J. Algebra 225 (2000) 517-521.
	
	\bibitem{ngoc_bien_hai_17}  N. K. Ngoc,  M. H. Bien,  B. X. Hai, Free subgroups in almost subnormal subgroups of general skew linear groups, Algebra i Analiz 28(5) (2016) 220-235, translation in St. Petersburg Math. J. 28(5) (2017) 707-717.
	
	\bibitem{passman_77} D. S. Passman, The algebraic structure of group rings, New York: Wiley- Interscience Publication, 1977.
	
	\bibitem{nassab14} M. Ramezan-Nassab and D. Kiani, Nilpotent and polycyclic-by-finite maximal subgroups of skew linear groups, J. Algebra 399 (2014) 269-276.
	
	\bibitem{rem-kiani} M. Ramezan-Nassab, D. Kiani, Some skew linear groups with Engel's condition,  J. Group Theory
	15 (2012) 529–541.
		
	\bibitem{rowen} L. H. Rowen, Polynomial Identities in Ring Theory, Academic Press, New York, 1980.
	
	\bibitem{shirvani-wehrfritz} M. Shirvani and B. A. F. Wehrfritz, Skew Linear Groups, Cambridge Univ. Press, 1986.
	
	\bibitem{stuth} C. J. Stuth, A generalization of the Cartan-Brauer-Hua Theorem, Proc. Amer. Math. Soc. 15(2) (1964) 211-217.
	
	\bibitem{wehrfritz_infinite-linear-group}  B. A. F. Wehrfritz, Infinite linear groups, Springer-Verlag, Berlin, 1973.
	
	\bibitem{wehrfritz84}  B. A. F. Wehrfritz, Soluble-by-periodic skew linear groups, Mathematical Proceedings of the Cambridge Philosophical Society 96(3) (1984) 379-3
	89.
	
	\bibitem{wehrfritz_89} B. A. F. Wehrfritz, Goldie subrings of Artinian rings generated by groups, Q. J. Math. Oxford 40 (1989) 501-512.
	
	\bibitem{wehrfritz_91} B. A. F. Wehrfritz, Crossed product criteria and skew linear groups, J. Algebra 141 (1991) 321–353.
	
	\bibitem{wehrfritz_90_criteriaII} B. A. F. Wehrfritz, Crossed product criteria and skew linear groups II, Michigan Math. J. 37 (1990) 293–303.
		
	\bibitem{wehrfritz_07} B. A. Wehrfritz, Normalizers of nilpotent subgroups of division rings, Q. J. Math. 58 (2007) 127-135.
	
\end{thebibliography}
\end{document}